\documentclass[10pt]{article}
\usepackage{graphicx} 
\usepackage{amsmath, amssymb, quiver, mathtools, amsthm, hyperref, dirtytalk, mathrsfs} 
\usepackage[english]{babel}
\usepackage{chngcntr}

\DeclareMathOperator{\Hom}{Hom}
\DeclareMathOperator{\id}{id}

\DeclareMathOperator{\MC}{MC}

\newtheorem{theorem}{Theorem}[section]
\newtheorem{proposition}[theorem]{Proposition}
\newtheorem{Corollary}[theorem]{Corollary}
\newtheorem{Lemma}[theorem]{Lemma}

\theoremstyle{remark}

\newtheorem{remark}{Remark}

\theoremstyle{definition}
\newtheorem{definition}[theorem]{Definition}

\title{An Operadic Generalization of the Gerstenhaber-Shack Theorem}
\author{Andy Yu}
\date{}

\begin{document}

\maketitle

\begin{abstract}
A simplicial cochain complex can be derived from a locally small poset by taking the nerve of the poset viewed as a category. We show that the simplicial cochain complex and a relative Hochschild cochain complex of the incidence algebra of the poset are isomorphic as operads with multiplications. This result implies that the hG-algebras derived from those operads are isomorphic, which is a generalization of the Gerstenhaber-Shack theorem. The isomorphism also induces a differential graded Lie algebra isomorphism, which we use to compute the moduli space of formal deformations of the incidence algebra.
\end{abstract}

\section*{Introduction} 
Let $P$ be a locally finite poset, which means for any $x_1, x_2 \in P$, the set $\{ x \in P \mid x_1 \leq x \leq x_2 \}$ is finite. We may regard $P$ as a category. Given $k$ a field of characteristic 0, we can construct an incidence algebra $kP$ \cite[section 2]{Webb2008_An_Introduction_to_the_Representations_and_Cohomology_of_Categories}. Let $S \subset R P$ be the subalgebra generated by all the identity morphisms in $P$. Gerstenhaber and Shack showed that $C^\bullet(P; k)$ is isomorphic to the relative Hochschild cochain complex $C^\bullet(R P, S; k P)$ \cite{GERSTENHABER1983143SimpIsHoch(Coho)}. 

For an operad $\mathcal{O}$ of a non-negatively graded $k$-vector space, we can choose an element $m \in \mathcal{O}(2)$ that is associative. A pair $( \mathcal{O}, m)$ is called an operad with multiplication. It has been shown that an operad with multiplication can be made into a hG-algebra, which contains a differential graded associative algebra (DGA) structure \cite{Voronov_Gerstenhaber_1995_hGA_And_Moduli_Space_Operad},  \cite{VG1995HigherOperationsHC}. 

We introduce the operad structures on the Hochschild and simplicial cochain complexes, in section 1, such that the DGA structures within coincide with the Hochschild and the simplicial cochain complexes. In section 2, we show that the isomorphism map described in \cite[page 137, 138]{GS.NATO} is in fact an isomorphism of two operads with multiplications, it is consequently a DGA isomorphism, which implies the Gerstenhaber-Shack theorem. In section 3, we compute the moduli space of deformations of the incidence algebra $kP$ using the property that Maurer-Cartan space functor preserves quasi-isomorphism of differential graded Lie algebras up to homotopy equivalence \cite{Getzler_2009_Lie_Theory_For_Nil_L_inf_Algebras}.

\paragraph{Conventions} We use $k$ to denote a field of characteristic 0. While some operads in this paper are operads of $k$-vector spaces, we suppress the phrase and simply say operads instead. In this paper, $\mathbb{Z}^+$ denotes the positive integers, and $\mathbb{N}$ contains 0.

Let $M = \bigoplus_{n \in \mathbb{Z}} M^n$ be a graded $k$-vector space, $f \in \Hom( M^{\otimes n}, M)$.
\begin{itemize} 
    \item For $x \in M^{n}$, let $|x| := n$.
    \item Let $M[n]$ be the shifted graded $k$-vector space defined by $(M[n])^m \coloneqq M^{m + n}$.
    \item For $x \in M^{n}$, let the same element be denoted by $sx \in (M[1])^{n - 1}$. Then $|sx| = n - 1$.
    \item The degree of a map $f: M \rightarrow M$ is defined by $\deg f := |f(v_1, \ldots, v_n)| - (|v_1| + \cdots + |v_n|)$.
    \item Let $\sigma \in S^{n}$ be a permutation of $n$ elements. Let $\kappa$ be the \textit{Koszul sign} defined by 
    \[
    x_1 x_2 \cdots x_n = \kappa x_{\sigma (1)} x_{\sigma (2)} \cdots  x_{\sigma (n)},
    \]
    for $x_i \in M$, $x_1 x_2 \cdots x_n$ in the free graded-commutative algebra $S^\bullet (M)$.
\end{itemize}


\paragraph{Acknowledgments} I am grateful to my advisor Alexander A. Voronov from the University of Minnesota, who introduced me to rational homotopy theory, operads, and deformation theory, for his continuous support, guidance, and belief in my capability. I am also thankful to Pranjal Dangwal from the University of Minnesota, who had fruitful discussions with me on deformation theory and rational homotopy theory, and partially accompanied me on this project.

\section{Operads and Homotopy G-Algebras}

\subsection{Preliminaries on Operads}

\begin{definition}
    A \textit{non-$\Sigma$ operad} is a collection of $k$-vector spaces $\mathcal{O}(n)$, $n \in \mathbb{N}$, along with operations $\gamma: \mathcal{O}(k) \otimes \mathcal{O}(n_1) \otimes \ldots \otimes \mathcal{O}(n_k) \rightarrow \mathcal{O}(n_1 + \ldots + n_k)$, called compositions, and a distinguished element $\id \in \mathcal{O}(1)$, called the identity, that satisfy the identity and associativity of composition axioms, see more detail in \cite{MarklStasheff2002OperadsIA}.
    A \textit{sub-operad} $\mathcal{P} \subseteq \mathcal{O}$ is an operad with respect to the restriction of the operations $\gamma_{|\mathcal{P}}$.
    A morphism of operads is a map that preserves the identity and compositions.
\end{definition}

For $x \in \mathcal{O}(0)$, we require $\gamma(x; \ ) := x$. Since we will only concern with non-$\Sigma$ operads, we refer to a non-$\Sigma$ operad simply as an operad.

\begin{definition}
    For $f \in \mathcal{O}(m)$, $m \geq 1$, and $g \in \mathcal{O}(n)$, $n \geq 0$, let $\circ_j: \mathcal{O}(m) \otimes \mathcal{O}(n) \rightarrow \mathcal{O}(m+n-1)$, $1 \leq j \leq m$, be a collection of maps be defined by 
    $f \circ_j g := \gamma (f; \mathrm{id}, \ldots, \mathrm{id}, g, \mathrm{id}, \ldots, \mathrm{id})$, where $g$ is placed in the $j$th input (excluding $f$) of $\gamma$.
\end{definition}

The maps $\circ_j$ come in handy. We can recover $\gamma$ by repeatedly composing the $\circ_j$ products.
 \begin{equation*} \label{operad composition map identity}
        \gamma ( f; f_1, \ldots, f_k) = (\cdots((f \circ_k f_k) \circ_{k-1} f_{k-1}) \cdots ) \circ_1 f_1.
    \end{equation*}
In addition, a map $F: \mathcal{O} \rightarrow \mathcal{P}$ is a morphism of operads if and only if for $f, g \in \mathcal{O}$, $F$ satisfies $F(f \circ_j g) = F(f) \circ_j F(g)$, and $F(\id_\mathcal{O}) = \id_\mathcal{P}$.

Now we exhibit relevant examples of operads.

\subsubsection*{Relative Hochschild Cochains As An Operad}
Let $A$ be an associative unital algebra over $k$. The endomorphism operad $\mathcal{E}nd_{A}$ is defined by $\mathcal{E}nd_{A} (n) = \Hom_R (A^{n}, A)$, $n \geq 0$. The operations $\gamma$ on $\mathcal{E}nd_{A}$ are defined by
\begin{multline} \label{Hochschild Gamma Def Eq}
    \gamma(f; f_1, \ldots, f_k) (a_1, \ldots, a_{n_1 + \cdots + n_k}) \\
    = f( f_1(a_1, \ldots, a_{n_1}), \ldots, f_k(a_{n_1 + \ldots n_{k-1}}, \ldots, a_{n_1 + \ldots + n_k})).
\end{multline}
Since $\Hom_R (A^n, A) = C^n(A; A)$, the Hochschild cochain complex $C^\bullet(A; A)$ is the same as $\mathcal{E}nd_{A}$, provided the Hochschild differential is ignored. Let $S$ be a subalgebra of $A$. Recall the definition of an $S$-relative Hochschild cochain complex \cite[page 32]{GS.NATO}. The following proposition shows that $C^\bullet (A, S; A)$ is also an operad.

\begin{proposition}
    The relative Hochschild cochain complex $C^\bullet (A, S; A)$ with $\gamma$ defined in \eqref{Hochschild Gamma Def Eq} is a sub-operad of $\mathcal{E}nd_{A}$.
\end{proposition}

\begin{proof}
    If $\gamma$ is closed in $C^\bullet (A, S; A)$, then its identity and associativity properties are inherited, so it suffices to show the closure of $\circ_j$ maps within $C^\bullet (A, S; A)$.
    Let $b \in S$, $f \in C^{p} (A, S; A)$, $g \in C^{q} (A, S; A)$. We need to verify the conditions 
    \begin{align} \label{Hochschild condition for composition in proof of sub-operad}
    \begin{split}
        (f \circ_j g)(ba_1,\ldots, a_{p + q - 1}) & = b (f \circ_j g)(a_1,\ldots, a_{p + q - 1}) \\
        (f \circ_j g) (\ldots, a_ib, a_{i+1},\ldots) & = (f \circ_j g) (\ldots, a_i, ba_{i+1},\ldots) \\
        (f \circ_j g) (a_1,\ldots, a_{p + q - 1} b) & = (f \circ_j g) (a_1,\ldots, a_{p + q - 1})b.
    \end{split}
    \end{align}
    We compute through definition and get
    \begin{multline*}
        (f \circ_j g)(ba_1,\ldots, a_{p + q - 1}) \\
        = f(ba_1,\ldots, a_{j - 1}, g(a_{j},\ldots, a_{j + q - 1}), a_{j + q}, \ldots,  a_{p + q - 1}) \\
        = bf(a_1,\ldots, a_{j - 1}, g(a_{j},\ldots, a_{j + q - 1}), a_{j + q}, \ldots,  a_{p + q - 1}) \\
        = b (f \circ_j g)(a_1,\ldots, a_{p + q - 1}).
    \end{multline*}
    The computation can be symmetrically carried out for the third line of \eqref{Hochschild condition for composition in proof of sub-operad}. For the second line, we have
    \begin{multline*}
        (f \circ_j g)(a_1,\ldots, a_{j-1}b, a_j, \ldots a_{p + q - 1})  \\
        = f(a_1,\ldots, a_{j - 1}b, g(a_{j},\ldots, a_{j + q - 1}), a_{j + q}, \ldots,  a_{p + q - 1}) \\
        = f(a_1,\ldots, a_{j - 1}, bg(a_{j},\ldots, a_{j + q - 1}), a_{j + q}, \ldots,  a_{p + q - 1}) \\
        = f(a_1,\ldots, a_{j - 1}, g(ba_{j},\ldots, a_{j + q - 1}), a_{j + q}, \ldots,  a_{p + q - 1}) \\
        = (f \circ_j g)(a_1,\ldots, a_{j-1}, ba_{j}, \ldots a_{p + q - 1}).
    \end{multline*}
    Similarly, the same holds for $i = j + q - 1$. We omit computations for other $i$'s as they are inherited from $C^\bullet (A, S; A)$. Thus $C^\bullet (A, S; A)$ is a sub-operad.

\end{proof}

We will use the cochain complex notation $C^\bullet(A, S; A)$, but regard it as an operad. Explicitly, for $n > 0$,  $C^n(A, S; A)$ consists of $f \in \mathrm{Hom}_k(A^n,M)$ satisfying
\begin{align*}\label{}
f(b a_1,\ldots, a_n) & = b f(a_1,\ldots, a_n) \\
f(a_1,\ldots, a_i b, a_{i+1},\ldots, a_n) & =  f(a_1,\ldots, a_i, b a_{i+1},\ldots, a_n) \\
f(a_1,\ldots, a_n b) & = f(a_1,\ldots, a_n) b 
\end{align*}
for $a_1,\ldots,a_n \in A$ and $b \in S$, and  $C^0(A, S; A) = \{ a \in A \mid a b = b a \mbox{ for all } b \in S \}$.

\subsubsection*{Simplicial Cochains As An Operad}
Let $C^\bullet(X; k)$ be the simplicial cochain complex of a simplicial set $X$. Let $\varphi, \varphi_1, \ldots, \varphi_k \in C^\bullet(X; k)$ with degrees $k, n_1, \ldots, n_k$. We define the $\gamma$ operations by 
\begin{multline*}
    \gamma( \varphi; \varphi_0, \ldots, \varphi_k) ( v_1, \ldots, v_{n_1 + \cdots + n_k}) = \varphi( v_0, v_{n_1}, v_{n_1 + n_2}, \ldots, v_{n_1 + \cdots + n_k}) \\
    \varphi_1 (v_0, v_1, \ldots, v_{n_1} ) \varphi_2 ( v_{n_1}, \ldots, v_{n_1 + n_2}) \cdots \varphi_k ( v_{n_1 + \cdots + n_{k-1}}, \ldots, v_{n_1 + \cdots + n_k}).
\end{multline*}
Then $\gamma$ determines an operad structure on the graded vector space $C^\bullet(X; k)$.

\subsection{The Homotopy G-Algebra of an Operad}

\begin{definition}
    A \textit{brace algebra} is a graded $k$-vector space $M = \bigoplus_{n \in \mathbb{Z}} M^n$ over $k$, together with a collection of multilinear braces $x\{ x_1, \ldots, x_m \}_{m \in \mathbb{N}}$ of degree $0$, such that
    \begin{equation*}
        x\{ \ \} = x,
    \end{equation*}
    \begin{multline*}
        x \{ x_1, \ldots, x_m\} \{y_1, \ldots, y_n\} \\
        = \sum (-1)^\varepsilon x\{ y_1, \ldots, y_{i_1}, x_1 \{ y_{i_1 + 1}, \ldots, y_{j_1} \},\\
        y_{j_1 + 1}, \ldots, y_{i_m}, x_m \{ y_{i_m + 1}, \ldots, y_n \},
    \end{multline*}
    where the sum takes over $0 \
    \leq i_1 \leq j_1 \leq \ldots \leq i_m \leq j_m \leq n$, and $\varepsilon := \sum_{p=1}^{m} (|x_p| - 1) \sum_{q=1}^{i_p} (|y_q| - 1) $.
\end{definition}

We use $E_k$ to denote the brace with $k$ inputs in the curly bracket:
\[
E_k(x; x_1, \dots, x_k) \coloneqq x\{x_1, \dots, x_k\},
\]
and $\{E_k \}_{k \in \mathbb{N}}$ to denote the collection of braces. We also use the notation
\[
f \circ g := f \{ g \}
\]
and call it the \emph{circle product}.

Given an operad $\{\mathcal{O}(n) \mid n \in \mathbb{N}\}$, consider the underlying graded vector space $\mathcal{O} \coloneqq \bigoplus_{n \in \mathbb{N}} \mathcal{O}(n)$. We can construct the brace algebra structure on the \emph{desuspension} $\mathcal{O}[1]$ of the graded vector space  $\mathcal{O}$  by defining the braces as follows, see  \cite[Proposition 1]{Voronov_Gerstenhaber_1995_hGA_And_Moduli_Space_Operad}:
    \begin{equation}
        \label{braces}
    sx \{ s x_1 ,\ldots, s x_n \} \coloneqq  
    \sum_{\mathclap{\sigma \in \mathrm{Sh}_{n, |x| - n} }} \kappa s \gamma ( s x; \sigma (s x_1, \ldots, s x_n, s \id, \ldots, s \id)),
    \end{equation}
    where $sx, sx_1, \dots, sx_n \in \mathcal{O}[1]$ and $\kappa$ is the Koszul sign for the composition of the permutations
    \begin{multline*}
        (s x_n, \ldots, s x_1, s, \ldots, s) \mapsto (s x_1, \ldots, s x_n, s, \ldots, s) \\
        \mapsto \sigma ( s x_1, \ldots, s x_n, s, \ldots, s).
    \end{multline*}


\begin{definition} \label{hGA definition}
    A \textit{homotopy G-Algebra} (hG-algebra) is a tuple $(M, \{ E_k \}, \cdot, d)$, where
    \begin{enumerate}
        \item $M$ is a graded $k$-vector space and $d$ is a differential on $V$ of degree +1.
        \item The triple $(M, \cdot, d)$ forms a DGA.
        \item $(M[1], \{ E_k \})$ is a brace algebra. 
        \item Let $x\{ s y_1, \ldots s y_n\}$ replace $s^{-1}( s x \{ s y_1, \ldots, s y_n\})$. Then we have 
        \[
        (s(x_1 \cdot x_2)) \{ sy_1, \ldots, sy_n \} = \sum_{i=1}^n \chi_i s (x_1 \{ sy_1, \ldots, sy_i\} \cdot x_2 \{ sy_{i+1}, \ldots, sy_n \}),
        \]
        where $\chi_i$ is the Koszul sign for the permutation 
        \[
        (x_1, x_2, sy_1, \ldots, sy_n) \mapsto (x_1, sy_1, \ldots, sy_i, x_2, sy_{i+1}, \ldots, sy_n).
        \]
        \item The braces satisfy the identity
        \begin{align*}
            d(& s x \{ s x_1, \ldots, s x_n \}) - d (s x) \{ s x_1, \ldots, s x_n\} \\
            +& \sum_{i = 1}^{n + 1} (-1)^{|s x|
            + \sum_{j = 1}^{i - 1} |s x_i|} s x \{ \ldots, d s x_i, \ldots\} \\
            = -&(-1)^{|s x_1| |x|} s ( x_1 \cdot x\{ s x_2, \ldots, s x_{n+1}\}) \\
            -& (-1)^{|x|}\sum_{i = 1}^{n+1}(-1)^{\sum_{j = 1}^{i - 1}|s x_i|} s x\{ s x_1, \ldots, s (x_{i} \cdot x_{i+1}), \ldots, s x_{n+1}\} \\
            +& (-1)^{|x| + \sum_{j = 1}^{n}|s x_i|} s( x\{ s x_1, \ldots, s x_n \} \cdot x_{n+1}),
        \end{align*}
        where the differential on $M[1]$ is defined by $d(sx) \coloneqq -sdx$.
    \end{enumerate}
\end{definition}

\begin{definition} \label{Operad with multiplication}
    A \textit{multiplication on an operad} $\mathcal{O}$ is an element $m \in \mathcal{O}(2)[1]$ such that $m \circ m = 0$.
    An \textit{operad with multiplication} is an operad $\mathcal{O}$ with a specific choice of multiplication.
\end{definition}

Given an operad with multiplication $(\mathcal{O}, m)$, we can construct an hG-algebra structure on $\mathcal{O}$ by the following, see \cite[Theorem 3]{Voronov_Gerstenhaber_1995_hGA_And_Moduli_Space_Operad}:

\begin{itemize}
    \item Define the braces on $\mathcal{O}[1]$ by \eqref{braces};
    \item Define the dot product on $\mathcal{O}$ by for $x, y \in \mathcal{O}$, $x \cdot y \coloneqq (-1)^{|x|} s^{-1} (m \{sx, sy\})$;
    \item Define the differential $d$ on $\mathcal{O}[1]$ by $d (sx) \coloneqq m \circ sx -(-1)^{|sx|} sx \circ m$. The differential $d$ on $\mathcal{O}$ is determined by $d(sx) = -s dx$.
\end{itemize}

The construction of the hG-algebra from an operad $\mathcal{O}$ may be viewed as a functor from the category of operads with multiplications to the category of hG-algebras. The same statement can be made about the DGA in \ref{hGA definition}.2. We describe them as \say{$\mathcal{O}$ viewed as a hG-algebra} or \say{$\mathcal{O}$ viewed as a DGA} when we are interested in theses structures of the operad $\mathcal{O}$.

\subsubsection*{Relative Hochschild and Simplicial Cochain Complexes Regarded As Operads With Multiplications}
Let $m_\mathrm{H} \in \mathcal{E}nd_{A}(2)[1]$ be defined by $m_\mathrm{H} (a_1, a_2) := a_1 a_2$ for $a_1, a_2 \in A$. Since $A$ is an associative algebra, $m_\mathrm{H}$ is a multiplication on the operad $\mathcal{E}nd_{A}$. 
In addition, we have $m_\mathrm{H} \in C^2 (A, S; A)$. So the pair $(C^\bullet (A, S; A), m_\mathrm{H})$ can also be regarded as an operad with multiplication.

Let $m_\mathrm{S} \in C^2(P;R)$ be the constant cochain defined by $m_\mathrm{S} (v_0, v_1, v_2) := 1_R$ for any 2-chain $(v_0, v_1, v_2)$. It is clear that $m_\mathrm{S} \circ m_\mathrm{S} = 0$. Then the pair $(C^\bullet(X; k), m_\mathrm{S})$ is an operad with multiplication. For the simplicity of notation, we will suppress the multiplication symbols.

\begin{remark}
    The circle product $f \circ g $ on a Hochschild cochain complex $C^\bullet(A; A)$ is in fact the Gerstenhaber composition product \cite[page 85]{GS.NATO}. The circle product on a simplicial cochain complex $C^\bullet(X; k)$ is the $\smile_1$ product defined by Steenrod \cite{SteenrodCup1}, up to the sign $(-1)^{|f|(|g| + 1 )}$.
\end{remark}

\begin{remark}[On Signs and Conventions of Cochain Complexes]
\label{Remark On Signs Conventions of Hochschild Cochain Complex}
    The dot product defined on the simplicial cochain complex $C^\bullet(X; k)$ by the homotopy G-algebra structure coincides with the standard cup product up to a sign: $x \cdot y = (-1)^{|x||y|} x \smile y$. Similarly, the differentials differ by the sign $(-1)^{|x|}$. Either way, we get isomorphic DGA structures on $C^\bullet(X; k)$. The same statement can be made about $C^\bullet(A, S; A)$, where the classical DGA structure on $C^\bullet(A, S; A)$ is defined in \cite{GS.NATO} and \cite{Hochschild_Original_Paper_CohomologyOfAssoAlg}.
\end{remark}

\begin{remark}
    One may observe that braces in a hG-algebra distinguish the first input of the braces from the rest, and ask if there is an extension of the concept that distinguishes the first $n$ inputs of operations of similar kinds, i.e. an algebra $A$ with a collection of maps $E_{p,q}: A^{\otimes p} \otimes A^{\otimes q} \rightarrow A$ that satisfies similar identities. Such extensions exists and are called \textit{Hirsch Algebras} or \textit{extended homotopy algebras}. Discussions of such algebras may be found at \cite{saneblidze2016filteredhirschalgebras} and \cite{kadeishvili2002cochainoperationsdefiningsteenrod}.
\end{remark}

\section{An Operadic Generalization of the Ger\-sten\-ha\-ber-Shack Theorem}

\begin{definition}[Category Algebra]
    Let $\mathcal{C}$ be a locally small category, $R$ be an unital commutative ring. The elements of the category algebra $R \mathcal{C}$ are formal sums $\sum_{f \in \hom(\mathcal{C})} r_f f$. The product on $R \mathcal{C}$ is defined on the generators by
    \begin{equation*}
    r_1 f \cdot r_2 g :=  \left\{ \begin{array}{rcl}
         &  (r_1 r_2) (f g) & \text{if $f$ and $g$ are composable}\\
         &  0 & \text{otherwise}.
    \end{array} \right.
    \end{equation*}
\end{definition}

Let a locally finite poset $P$ be viewed as a category: the objects are elements, and the $\leq$ relations are morphisms. The axioms of category are satisfied by the transitivity and reflexivity of the relation $\leq$. Let $C^\bullet(P; k)$ be the simplicial cochain complex associated to the nerve of $P$. Let $kP$ denote the category algebra of the category $P$ over $k$, also known as the \emph{incidence algebra} of $P$ as a poset. Let $S \subseteq kP$ be the subalgebra generated by the identity morphisms in $P$. For $i, j \in P$, if $i \leq j$, we denote the unique morphism from $i$ to $j$ by $E^{ij}$.

Let $\Phi: C^\bullet(P; k) \rightarrow C^\bullet(kP, S; kP)$ be defined by 
\begin{align*}
    (\Phi f)&:=\sum_{i \in S} f(i)E^{ii} & \mbox{ for }f \in C^\bullet(P; k) (0),\\
    (\Phi f)(E^{i_0i_1},\ldots E^{i_{n-1}i_n})&:=f(i_0,\ldots,i_n)E^{i_0i_n} & \mbox{ for }f \in C^\bullet(P; k) (n), n>0.
\end{align*}

\begin{theorem}[Main Theorem] \label{operad isomorphism}
    The map $\Phi$ is an isomorphism of operads with multiplication.
\end{theorem}
\begin{proof}
    We first show that $\Phi$ is an isomorphism of operads. 
    For each $n \in \mathbb{N}$, the map $\Phi$ is clearly an isomorphism of bases of $C^n(P; k)$ and $C^n(kP, S; kP)$ for each $n$. Let $f \in C^p(P; k)$, $g \in C^q(P; k)$. It suffices to show that
    \[
    \Phi(f \circ_j g) = \Phi f \circ_j \Phi g.
    \]
    For $(E^{i_0 i_1}, \ldots, E^{i_{p+q-2} i_{p+q-1}}) \in C^{p + q - 1}$, we have 
    \begin{align*}
        & \Phi(f \circ_j g)(E^{i_0 i_1}, \ldots, E^{i_{p+q-2} i_{p+q-1}}) = f \circ_j g (i_0, \ldots, i_{p+q-1}) E^{i_0 i_{p+q-1}} \\
        & =  f( i_0, \ldots, i_{j-1}, i_{j+q-1}, \ldots, i_{p+q-1}) \cdot g( i_{j-1}, \ldots, i_{j+q-1}) E^{i_0 i_{p+q-1}}\\
        & =  \Phi f (E^{i_0 i_1}, \ldots, E^{i_{j-2} i_{j-1}},  E^{i_{j-1},  i_{j+q-1}}, E^{i_{j+q-1} i_{j+q}}, \ldots, E^{i_{p+q-2} i_{p+q-1}}) \\
        & \qquad \cdot g(i_{j-1}, \ldots, i_{j+q-1}) \\
        &  = \Phi f (E^{i_0 i_1}, \ldots, E^{i_{j-2} i_{j-1}}, g(i_{j-1}, \ldots, i_{j+q-1}) E^{i_{j-1}, i_{j+q-1}}, \\  
        & \qquad E^{i_{j+q-1} i_{j+q}}, \ldots, E^{i_{p+q-2} i_{p+q-1}} ) \\
        & = \Phi f (E^{i_0 i_1}, \ldots, E^{i_{j-2} i_{j-1}}, \Phi     g (E^{i_{j-1},  i_{j}} ,\ldots, E^{i_{j+q-2},  i_{j+q-1}}), \\ 
        & \qquad E^{i_{j+q-1} i_{j+q}}, \ldots, E^{i_{p+q-2} i_{p+q-1}}) \\
        & = \Phi f \circ_j \Phi g (E^{i_0 i_1}, \ldots, E^{i_{p+q-2} i_{p+q-1}}).
    \end{align*}t we show that $\Phi(m_\mathrm{S}) = m_{\mathrm{H}}$. For any $(E^{i_0 i_1}, E^{i_1 i_2}) \in RP^{k2}$, we have 
    \begin{equation*}
        \Phi(m_\mathrm{S}) (E^{i_0 i_1}, E^{i_1 i_2}) = m_\mathrm{S} (i_0, i_1, i_2) E^{i_0 i_2} = E^{i_0 i_2} = m_\mathrm{H} (E^{i_0 i_1}, E^{i_1 i_2}).
    \end{equation*}
    Since $\Phi (m_\mathrm{S})$ and $m_\mathrm{H}$ agree on the basis of $kP^2$, we conclude $\Phi (m_\mathrm{S}) = m_\mathrm{H}$.
\end{proof}

The next two corollaries follow immediately from Theorem \ref{operad isomorphism}.

\begin{Corollary} \label{hGA isomorphism Corollary}
    The map $\Phi: C^\bullet(P; k) \rightarrow C^\bullet(kP, S; kP)$ is an isomorphism of hG-algebras. \qed
\end{Corollary}
\begin{Corollary}[Gerstenhaber-Shack Theorem] \label{Gerstenhaber-Shack Theorem}
   The map $\Phi: C^\bullet(P; k) \rightarrow C^\bullet(kP, S; kP)$ is an isomorphism of DGA's. \qed
\end{Corollary}
t brace of a hG-algebra $G$ is pre-Lie with respect to the degree on ${G[1]}$. That is, for $f, g \in G[1]$, $[f, g] := f \circ g  - (-1)^{|f| |g|} g \circ f$ is a Lie bracket. Since the bracket is completely construc determinedrace operation, the differential graded Lie algebras (DGLA) derived from isomorphic operads with multiplications are isomorphic, we obtain the following corollary.
\begin{Corollary} \label{DGLA isomorphism Corollary}
    $\Phi: (C^{\bullet}(P; R)[1], ; k -], d) \rightarrow (C^{\bullet} (RP, S; RPk[1], [k, -], d)$ is a DGLA isomorphism. \qed
\end{Corollary}section{Moduli Space of Deformations of Incidence Algebra}



The Maurer-Cartan space $\MC_\bullet(\mathfrak{g})$, also known as the nerve of a DGLA $\mathfrak{g}$, is invariant under quasi-isomorphisms up to homotopy equivalence \cite{CostelloPartitionFunctionOfATFTzbMATH05665533}, \cite{Getzler_2009_Lie_Theory_For_Nil_L_inf_Algebras}. We introduce the Maurer-Cartan space and use it compute the formal deformations of the incidence algebra $kP$ using the big Witt vectors $W$. We show the moduli space of formal deformations of $kP$, $\pi_0(\MC_\bullet(\lambda C^\bullet(kP; kP)[[\lambda]]))$, is isomorphic $H^2(P; W)$.

\begin{definition}
    Let $\Omega_\bullet$ be the simplicial differential graded commutative algebra (DGCA) defined by
    \[
    \Omega_n := \frac{R[ t_0, \ldots, t_n, dt_0, \ldots, dt_n ]}{\big( \sum t_i = 1, \sum dt_i = 0 \big)}, \qquad \deg(t_i) = 0, \deg(dt_i) = 1.
    \]
\end{definition}

Let $\mathfrak{g}$ be a DGLA over $k$. Let the \textit{Maurer-Cartan set} of $\mathfrak{g}$ be defined by
\[
\MC(\mathfrak{g}) := \{ a \in \mathfrak{g}^1 \mid d a + \frac{1}{2} [a, a] = 0 \}.
\]
The equation $d a + \frac{1}{2} [a, a] = 0$ is called the \textit{Maurer-Cartan equation}. Let us assume that $\mathfrak{g}^0$ is a pro-nilpotent Lie algebra over $k$. We associate with it the group $\exp(\mathfrak{g}^0)$ of formal symbols
$\exp(x)$, $x \in \mathfrak{g}$, with multiplication given by the Campbell-Baker-Hausdorff formula \cite[Section 1.4]{kontsevich-soibelmanDeformationTheoryI}. The group $\exp(\mathfrak{g}^0)$ is called the \emph{gauge transformation group} and has a natural action on $\MC(\mathfrak{g})$. See more discussions at \cite[Chapter 6]{Manetti-Lie-methods-in-deformation-theory} and \cite[Section 2]{GoldmanWilliamMillsonDeformationAndCOmpactKahlerManifolds}.

The \textit{derived $($Maurer-Cartan$)$ moduli space}, or \textit{nerve}, of $\mathfrak{g}$ is defined by 
\[
\MC_\bullet(\mathfrak{g}) := \MC( \mathfrak{g} \otimes_R \Omega_\bullet).
\]
Let
\[
\mathscr{MC}(\mathfrak{g}) := \pi_0 (\MC_\bullet (\mathfrak{g}))
\]
be the \emph{$($Maurer-Cartan$)$ moduli space}, which can be identified with the orbit space $\MC (\mathfrak{g}) / \exp(\mathfrak{g}^0)$ \cite{Getzler_2009_Lie_Theory_For_Nil_L_inf_Algebras}. This is viewed here as a pointed set with base point 0.

\begin{definition}
    A \textit{formal deformation} on an associative algebra $A$ is an associative $k[[\lambda]]$-linear map $F \colon A[[\lambda]] \otimes_{k[[\lambda]]} A[[\lambda]] \rightarrow A[[\lambda]]$ of the form
    \[
    F = m_A + \lambda F_1 + \lambda^2 F_2 + \dots,
    \]
where $m_A : A \otimes_R A \rightarrow A$ is the original product on $A$ and, for each $n \ge 1$, $F_n \colon A \otimes_R A \to A$ is an $k$-linear map, extended by $\lambda$-linearity, such that and $F(1, a) = F(a, 1) = a$ for all $a \in A$.
    Two deformations $F$, $F'$ are \textit{equivalent} if there is an $k[[\lambda]]$-algebra isomorphism $f \colon (A[[\lambda]],F') \to (A[[\lambda]],F)$ of the form$ f = \id_A + \sum_{n \in \mathbb{Z}^+} f_n \lambda^n$.
\end{definition}

Any formal deformation of the incidence algebra $A =kP$ may be written as
\begin{equation}
    \label{General deformation}
    F = m_{kP} + \sum_{n \in \mathbb{Z}^+} F_{n} \lambda^n \in C^2 (kP; kP)[[\lambda]].
\end{equation}
with $F_n \in C^2 (kP; kP)$.
Then the associativity condition on $F$ may be written as
\[
F(F(a,b),c) = F(a,F(b,c)) \qquad \text{for all $a,b,c \in kP$}.
\]
This is equivalent to the condition that $F \in \mathcal{E}nd_{kP} (2) [[\lambda]]$ is an operad multiplication:
\begin{equation}
    \label{assoc}
   F \circ F = 0 \quad \text{or, equivalently,} \quad [ F, F ] = 0,
\end{equation}
where the circle product and bracket are extended from $C^2(kP; kP)$ to $C^2 (kP; \linebreak[0] kP)[[\lambda]]$ by $\lambda$-linearity.
Since the Hochschild differential may be expressed through the Gerstenhaber bracket: $d = [m_A, -]$, the formal deformation $F$ corresponds to a formal solution $\tilde{F} \coloneqq \sum_{n \in \mathbb{Z}^+} F_{n} \lambda^n$ of the Maurer-Cartan equation
\[
d\tilde{F} + \frac{1}{2} [\tilde{F}, \tilde{F}] = 0.
\]
The following correspondence has been known \cite[ch. 4]{MarklMartinDefTheoryOfAlgebrasanddiagrams}.
\begin{equation}
\label{form-def}
\left\{ \begin{array}{cc}
     \mbox{Equivalence classes of}  \\
     \mbox{formal deformations of } kP  
\end{array} \right\} \cong \mathscr{MC} ( \lambda C^\bullet (kP; kP)[[ \lambda]]).
\end{equation}

The next lemma shows that a formal deformation $F$ of $kP$ implies the existence of an equivalent deformation $F'$ such that $F_n' \in C^2(kP, S; kP)$ for $n > 0$.

\begin{Lemma} \label{Relative Quasi-isomorphic to Absolute implies same deformation theory}
\begin{sloppypar}
        There is a homotopy equivalence of derived Maurer-Cartan spaces
    \(
    \MC_\bullet(\lambda C^\bullet(kP; kP)[[\lambda]]) \simeq \MC_\bullet(\lambda C^\bullet(kP, S; kP))[[\lambda]]).
    \)
\end{sloppypar}
\end{Lemma}

\begin{proof}
\begin{sloppypar}
     The inclusion $i: C^\bullet(kP, S; kP) \hookrightarrow C^\bullet(kP; kP)$ is a DGLA quasi-isomorphism \cite[page 137]{GS.NATO}. Then $i \otimes \id_{\lambda R [[\lambda]]}: \lambda C^\bullet(kP, S; kP) [[\lambda]] \hookrightarrow \lambda C^\bullet(kP; kP) [[\lambda]]$ is also a DGLA quasi-isomorphism. The ring $\lambda k[[\lambda]] = \varprojlim_n \lambda R[\lambda] / \lambda^n$ is pro-nilpotent, so the DGLAs $\lambda C^\bullet(kP, S; kP)[[\lambda]]$ and $\lambda C^\bullet (kP; kP)[[\lambda]]$ are also pro-nilpotent. 
     Then the standard results of the homotopy equivalence of the nerves for quasi-isomorphic pro-nilpotent DGLAs, explicitly
     \cite[Lemma 5.3.1]{CostelloPartitionFunctionOfATFTzbMATH05665533}, see also \cite[Proposition 4.9]
     {Getzler_2009_Lie_Theory_For_Nil_L_inf_Algebras}, implies the statement of the lemma. \qedhere
\end{sloppypar}
\end{proof}

Corollary \ref{DGLA isomorphism Corollary} and Lemma \ref{Relative Quasi-isomorphic to Absolute implies same deformation theory} implies the following corollary.

\begin{Corollary} \label{connecting corollary}
    There is a homotopy equivalence of derived Maurer-Cartan spaces $\MC_\bullet(\lambda C^\bullet(kP; kP)[[\lambda]]) \simeq \MC_\bullet(\lambda C^\bullet(P; k))[[\lambda]])$. \qed
\end{Corollary}

We would like to conclude with the identification of the Maurer-Cartan moduli space $\mathscr{MC} (\lambda C^\bullet(P; k)[[\lambda]])$.

Corollary \ref{connecting corollary} implies any solution of the Maurer-Cartan equation in the DGLA $\lambda C^\bullet(P; k)[[\lambda]]$
can be identified with an element
\begin{equation}
    \label{deformation}
\omega = m_\mathrm{S} + \sum_{n \in \mathbb{Z}^+} \omega_{n} \lambda^n \in C^2 (P; k)[[\lambda]] = C^2 (P; k[[\lambda]])
\end{equation}
that satisfies \ref{assoc}. That is, $\tilde{\omega} := \omega - m_\mathrm{S}$ satisfies the Maurer-Cartan equation 
\begin{equation}
\label{MCE}
d\tilde{\omega} + \frac{1}{2} [\tilde{\omega}, \tilde{\omega}] = 0.
\end{equation}
By $\lambda$-linearity, we have for $\omega$ and $\rho \in C^\bullet (P; k)[[\lambda]]$,
\begin{equation*}
    [\omega, \rho] = \sum_{n \in \mathbb{Z}^+} \sum_{p + q = n} [\omega_p, \rho_q] \lambda^n \qquad \text{and} \qquad d \omega := \sum_{n \in \mathbb{Z}^+} (d \omega_n) \lambda^n.
\end{equation*}
Collecting the coefficients in equation \eqref{MCE} by each power of $\lambda$, we get the collection of equations
\begin{equation}
\label{MCE-term}
    d \omega_n + \frac{1}{2}\sum_{p + q = n} [\omega_p, \omega_q] = 0, \qquad n \in \mathbb{Z}^+.
\end{equation}
Since all the $\omega_n$'s are 2-cochains, $[\omega_p, \omega_q] = \omega_p \circ \omega_q  - (-1)^{|\omega_p| |\omega_q|} \omega_q \circ \omega_p = \omega_p \circ \omega_q  + \omega_q \circ \omega_p$. We rearrange the terms in the sum and obtain 
\begin{equation} \label{MCE term with circ}
    d \omega_n + \sum_{p + q = n} \omega_p \circ \omega_q = 0.
\end{equation}

Let $W$ be the abelian multiplicative group of the big Witt ring $1_R + \lambda k[[\lambda]]$. Then every $\omega \in C^2(P; W)$ is of the same form \eqref{deformation}. Then $d \omega = 1_W$ if and only if $(d_0 \omega) (d_2 \omega) = (d_1 \omega) (d_3 \omega)$. Term by term, this equation is equivalent to the collection of equations
\begin{equation}
    d_0 \omega_n + d_2 \omega_n + \sum_{p + q = n} (d_0 \omega_p)(d_2 \omega_q) = d_1 \omega_n + d_3 \omega_n + \sum_{p + q = n} (d_1 \omega_p) (d_3 \omega_q).
\end{equation}
The terms involving $\omega_n$ form the differential $d\omega_n$ in $C^\bullet(P;R)$. Moving everything to the left-hand side, we obtain 
\begin{equation}
    d \omega_n + \sum_{p + q = n} (d_0 \omega_p) (d_2 \omega_q) - (d_1 \omega_p)(d_3 \omega_q) = d \omega_n + \sum_{p + q = n} \omega_p \circ \omega_q  = 0,
\end{equation}
which is the same as \eqref{MCE term with circ}. Thus, we see that the associativity equation \eqref{assoc} for $\omega$ as in \eqref{deformation} is equivalent to $d \omega = 1_W$ for the same $\omega$ considered as a cochain in $C^2(P; W)$. So we obtain the following lemma.

\begin{Lemma} \label{Kernel of H2(P; W) is formal deformation}
    An element $\omega \in \lambda C^2(P; k)[[\lambda]]$ is an MC element 
    if and only if $\omega \in Z^2(P; W)$. \qed
\end{Lemma}

\begin{Lemma} \label{Coboundary of H2(P; W) models formal deformation equivalence}
    Two
    Maurer-Cartan elements $\omega$ and $\omega' \in \lambda C^2(P; k)[[\lambda]]$ are equivalent if and only if they are cohomologous in $C^2(P; W)$.
\end{Lemma}
\begin{proof}
    Suppose $\omega$ and $\omega'$ are equivalent. Then there is $\varphi \in C^1(P; W)$ such that $\varphi  (\omega') = \omega ( \varphi \otimes \varphi)$. Rewriting this equation with the operad notation we obtain $(d_1 \varphi) \omega_q = (d_0 \varphi) (d_1 \varphi) \omega'$, where $d_i$'s are face maps. Rearranging the terms, we get $\omega = (d \varphi ) \omega'$, which means $\omega$ and $\omega'$ are cohomologous. Mirroring the process yeilds the proof for the reverse direction.
\end{proof}

\begin{theorem}
\label{Simplicial moduli space computation theorem}
The Maurer-Cartan moduli space $\mathscr{MC} (\lambda C^\bullet(P; k)[[\lambda]])$ is isomorphic to the cohomology group $H^2 (P; W)$ of the poset $P$ with coefficients in the group of big Witt vectors: \[
\mathscr{MC} (\lambda C^\bullet(P; k)[[\lambda]]) \cong H^2 (P; W).
\]
\end{theorem}
\begin{proof}
    Lemmas \ref{Kernel of H2(P; W) is formal deformation} and \ref{Coboundary of H2(P; W) models formal deformation equivalence} imply the statement.
\end{proof}

\begin{remark}
    The argument we used to prove the above theorem is essentially contained in Gerstenhaber and Schack's paper \cite{GERSTENHABER198653_RelativeHochschild_Cohomology}, except that they were not aware of the DGLA structure on $C^\bullet(P; k)$ and thereby could not discuss its Maurer-Cartan moduli space. They were working directly with formal deformations of the incidence algebra $kP$ and used the fact that the DGA isomorphism (here: Corollary \ref{Gerstenhaber-Shack Theorem}) is compatible with the $\smile_1$ and circle products. The goal of their computation was the following result.
\end{remark}

\begin{Corollary}[Gerstenhaber-Shack \cite{GERSTENHABER198653_RelativeHochschild_Cohomology}]
The set of equivalence classes of formal deformations of the incidence algebra $kP$ may be naturally identified with the cohomology group $H^2(P;W)$ of the poset $P$.
\end{Corollary}

\begin{proof}
    In view of the identification \eqref{form-def}, we need to present a natural isomorphism
    \[
    \mathscr{MC} (\lambda C^\bullet(kP; kP)[[\lambda]])) \cong H^2 (P; W),
    \]
    which follows from Theorem \ref{Simplicial moduli space computation theorem} and Corollary \ref{connecting corollary}.
\end{proof}

\bibliographystyle{amsalpha}
\bibliography{paper}

\providecommand{\bysame}{\leavevmode\hbox to3em{\hrulefill}\thinspace}
\providecommand{\MR}{\relax\ifhmode\unskip\space\fi MR }
\providecommand{\MRhref}[2]{%
  \href{http://www.ams.org/mathscinet-getitem?mr=#1}{#2}
}
\providecommand{\href}[2]{#2}
\begin{thebibliography}{MSS02}

\bibitem[Cos09]{CostelloPartitionFunctionOfATFTzbMATH05665533}
K.~Costello, \emph{The partition function of a topological field theory}, J. Topol. \textbf{2} (2009), no.~4, 779--822.

\bibitem[Get09]{Getzler_2009_Lie_Theory_For_Nil_L_inf_Algebras}
E.~Getzler, \emph{Lie theory for nilpotent {{\(L_{\infty}\)}}-algebras}, Ann. Math. (2) \textbf{170} (2009), no.~1, 271--301.

\bibitem[GM88]{GoldmanWilliamMillsonDeformationAndCOmpactKahlerManifolds}
W.~M. Goldman and J.~J. Millson, \emph{The deformation theory of representations of fundamental groups of compact {K{\"a}hler} manifolds}, Publ. Math., Inst. Hautes {\'E}tud. Sci. \textbf{67} (1988), 43--96.

\bibitem[GS83]{GERSTENHABER1983143SimpIsHoch(Coho)}
M.~Gerstenhaber and S.~D. Schack, \emph{Simplicial cohomology is {Hochschild} cohomology}, J. Pure Appl. Algebra \textbf{30} (1983), 143--156.

\bibitem[GS86]{GERSTENHABER198653_RelativeHochschild_Cohomology}
\bysame, \emph{Relative {Hochschild} cohomology, rigid algebras, and the {Bockstein}}, J. Pure Appl. Algebra \textbf{43} (1986), 53--74.

\bibitem[GS88]{GS.NATO}
\bysame, \emph{Algebraic cohomology and deformation theory}, Deformation theory of algebras and structures and applications, {Nato} {Adv}. {Study} {Inst}., {Castelvecchio}-{Pascoli}/{Italy} 1986, {Nato} {ASI} {Ser}., {Ser}. {C} 247, 11-264 (1988)., 1988.

\bibitem[GV95]{Voronov_Gerstenhaber_1995_hGA_And_Moduli_Space_Operad}
M.~Gerstenhaber and A.~A. Voronov, \emph{Homotopy {{\(G\)}}-algebras and moduli space operad}, Int. Math. Res. Not. \textbf{1995} (1995), no.~3, 141--153.

\bibitem[Hoc46]{Hochschild_Original_Paper_CohomologyOfAssoAlg}
G.~Hochschild, \emph{On the cohomology theory for associative algebras}, Ann. Math. (2) \textbf{47} (1946), 568--579.

\bibitem[Kad03]{kadeishvili2002cochainoperationsdefiningsteenrod}
T.~Kadeishvili, \emph{Cochain operations defining {Steenrod}-{{\(\smile_i\)}}-products in the bar construction}, Georgian Math. J. \textbf{10} (2003), no.~1, 115--125.

\bibitem[KS02]{kontsevich-soibelmanDeformationTheoryI}
M.~Kontsevich and Y.~Soibelman, \emph{Deformation theory}, vol.~1, Unpublished book, 2002, \url{https://www.math.ksu.edu/~soibel/Book-vol1.ps}.

\bibitem[Man22]{Manetti-Lie-methods-in-deformation-theory}
M.~Manetti, \emph{Lie methods in deformation theory}, Springer Monogr. Math., Singapore: Springer, 2022.

\bibitem[Mar12]{MarklMartinDefTheoryOfAlgebrasanddiagrams}
M.~Markl, \emph{Deformation theory of algebras and their diagrams.}, CBMS Reg. Conf. Ser. Math., vol. 116, Providence, RI: American Mathematical Society (AMS); Washington, DC: Conference Board of the Mathematical Sciences (CBMS), 2012.

\bibitem[MSS02]{MarklStasheff2002OperadsIA}
M.~Markl, S.~Shnider, and J.~Stasheff, \emph{Operads in algebra, topology and physics}, Math. Surv. Monogr., vol.~96, Providence, RI: American Mathematical Society (AMS), 2002.

\bibitem[San16]{saneblidze2016filteredhirschalgebras}
S.~Saneblidze, \emph{Filtered {H}irsch algebras}, Trans. A. Razmadze Math. Inst. \textbf{170} (2016), no.~1, 114--136.

\bibitem[Ste47]{SteenrodCup1}
N.~E. Steenrod, \emph{Products of cocycles and extensions of mappings}, Ann. Math. (2) \textbf{48} (1947), 290--320.

\bibitem[VG95]{VG1995HigherOperationsHC}
A.~A. Voronov and M.~Gerstenhaber, \emph{Higher operations on the {Hochschild} complex}, Funct. Anal. Appl. \textbf{29} (1995), no.~1, 1--5.

\bibitem[Web07]{Webb2008_An_Introduction_to_the_Representations_and_Cohomology_of_Categories}
P.~Webb, \emph{An introduction to the representations and cohomology of categories.}, Group representation theory. Based on the research semester ``Group representation theory'', Lausanne, Switzerland, January -- June, 2005., Boca Raton, FL: CRC Press; Lausanne: EPFL Press, 2007, pp.~149--173.

\end{thebibliography}

\end{document}